\documentclass{article}%
\usepackage{amsmath}
\usepackage{amsfonts}
\usepackage{amssymb}
\usepackage{graphicx}%
\newtheorem{theorem}{Theorem} [section]

\newtheorem{corollary}[theorem]{Corollary}

\newtheorem{definition}[theorem]{Definition}

\newtheorem{lemma}[theorem]{Lemma}

\newtheorem{proposition}[theorem]{Proposition}

\newenvironment{proof}[1][Proof]{\noindent\textbf{#1.} }{\ \rule{0.5em}{0.5em}}
\begin{document}

\author{Vadim E. Levit\\Department of Computer Science and Mathematics\\Ariel University Center of Samaria, ISRAEL\\levitv@ariel.ac.il
\and Eugen Mandrescu\\Department of Computer Science\\Holon Institute of Technology, ISRAEL\\eugen\_m@hit.ac.il}
\title{Graph Operations that are Good for Greedoids}
\date{}
\maketitle

\begin{abstract}
$S$ is a \textit{local maximum stable set} of a graph $G$, and we write
$S\in\Psi(G)$, if the set $S$ is a maximum stable set of the subgraph induced
by $S\cup N(S)$, where $N(S)$ is the neighborhood of $S$. In \cite{LevMan2} we
have proved that $\Psi(G)$ is a greedoid for every forest $G$. The cases of
bipartite graphs and triangle-free graphs were analyzed in \cite{LevMan45} and
\cite{LevMan07}, respectively.

In this paper we give necessary and sufficient conditions for $\Psi(G)$ to
form a greedoid, where $G$ is:

\textit{(a)} the disjoint union of a family of graphs;

\textit{(b)} the Zykov sum of a family of graphs;

\textit{(c)} the corona $X\circ\{H_{1},H_{2},...,H_{n}\}$ obtained by joining
each vertex $x$ of a graph $X$ to all the vertices of a graph $H_{x}$.


\end{abstract}

\section{Introduction}

Throughout this paper $G=(V,E)$ is a simple (i.e., a finite, undirected,
loopless and without multiple edges) graph with vertex set $V=V(G)$ and edge
set $E=E(G).$ If $X\subset V$, then $G[X]$ is the subgraph of $G$ spanned by
$X$. By $G-W$ we mean the subgraph $G[V-W]$, if $W\subset V(G)$. We also
denote by $G-F$ the partial subgraph of $G$ obtained by deleting the edges of
$F$, for $F\subset E(G)$, and we write shortly $G-e$, whenever $F$ $=\{e\}$.

The \textit{neighborhood} of a vertex $v\in V$ is the set $N_{G}(v)=\{w:w\in
V$ \ \textit{and} $vw\in E\}$. We denote the \textit{neighborhood} of
$A\subset V$ by $N_{G}(A)=\{v\in V-A:N(v)\cap A\neq\varnothing\}$ and its
\textit{closed neighborhood} by $N_{G}[A]=A\cup N(A)$, or shortly, $N(A)$ and
$N[A]$, if no ambiguity.

$K_{n},P_{n},C_{n}$ denote respectively, the complete graph on $n\geq1$
vertices, the chordless path on $n\geq2$ vertices, and the chordless cycle on
$n\geq3$ vertices, respectively.

A \textit{stable} set in $G$ is a set of pairwise non-adjacent vertices. A
stable set of maximum size will be referred to as a \textit{maximum stable
set} of $G$, and the \textit{stability number }of $G$, denoted by $\alpha(G)$,
is the cardinality of a maximum stable set in $G$. In the sequel, by
$\Omega(G)$ we denote the set of all maximum stable sets of the graph $G$.

Any stable set $S$ is maximal (with respect to set inclusion) in $G[N[S]]$,
but is not necessarily, a maximum one. A set $A\subseteq V(G)$ is a
\textit{local maximum stable set} of $G$ if $A$ is a maximum stable set in the
subgraph induced by $N[A]$, i.e., $A\in\Omega(G[N[A]])$, \cite{LevMan2}. Let
$\Psi(G)$ stand for the set of all local maximum stable sets of $G$.

Clearly, every stable set containing only pendant vertices belongs to
$\Psi(G)$. Nevertheless, there exist local maximum stable sets that do not
contain pendant vertices. For instance, $\{e,g\}\in\Psi(G)$, where $G$ is the
graph from Figure \ref{fig101}.\begin{figure}[h]
\setlength{\unitlength}{1.0cm} \begin{picture}(5,1.5)\thicklines
\multiput(5,0)(1,0){5}{\circle*{0.29}}
\multiput(7,1)(1,0){2}{\circle*{0.29}}
\put(5,0){\line(1,0){4}}
\put(7,1){\line(1,0){1}}
\put(7,0){\line(0,1){1}}
\put(8,1){\line(1,-1){1}}
\put(5,0.35){\makebox(0,0){$a$}}
\put(6,0.35){\makebox(0,0){$b$}}
\put(6.75,0.3){\makebox(0,0){$c$}}
\put(8,0.35){\makebox(0,0){$d$}}
\put(6.75,1.2){\makebox(0,0){$g$}}
\put(8.3,1.2){\makebox(0,0){$f$}}
\put(9,0.3){\makebox(0,0){$e$}}
\put(4,0.5){\makebox(0,0){$W$}}
\end{picture}\caption{A graph having {various local maximum stable sets}.}%
\label{fig101}%
\end{figure}
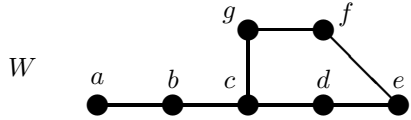

The following theorem concerning maximum stable sets in general graphs, due to
Nemhauser and Trotter Jr. \cite{NemhTro}, shows that some stable sets can be
enlarged to maximum stable sets.

\begin{theorem}
\cite{NemhTro}\label{th1} Every local maximum stable set of a graph is a
subset of a maximum stable set.
\end{theorem}

Nemhauser and Trotter Jr. interpret this assertion as a sufficient local
optimality condition for a binary integer programming formulation of the
weighted maximum stable set problem, and use it to prove an outstanding result
claiming that integer parts of solutions of the corresponding linear
programming relaxation maintain the same values in the optimal solutions of
its binary integer programming counterpart. In other words, it means that a
well-known branch-and-bound heuristic for general integer programming problems
turns out to be an exact algorithm solving the weighted maximum stable set problem.

The graph $W$ from Figure \ref{fig101} has the property that every $S\in
\Omega(W)$ contains some local maximum stable set, but these local maximum
stable sets are of different cardinalities: $\{a,d,f\}\in\Omega(W)$ and
$\{a\},\{d,f\}\in\Psi(W)$, while for $\{b,e,g\}\in\Omega(W)$ only
$\{e,g\}\in\Psi(W)$.

However, there exists a graph $G$ satisfying $\Psi(G)=\Omega(G)$, e.g.,
$G=C_{n}$, for $n\geq4$.

A greedoid is a set system generalizing the notion of a matroid.

\begin{definition}
\cite{BjZiegler}, \cite{KorLovSch} A \textit{greedoid} is a pair
$(E,\mathcal{F})$, where $\mathcal{F}\subseteq2^{E}$ is a non-empty set system
satisfying the following conditions:

\setlength {\parindent}{0.0cm}\textit{Accessibility}:

\hspace{1in}for every non-empty $X\in\mathcal{F}$ there is an $x\in X$ such
that $X-\{x\}\in\mathcal{F}$;

\setlength {\parindent}{0.0cm}\textit{Exchange}:

\hspace{1in}for $X,Y\in\mathcal{F},\left\vert X\right\vert =\left\vert
Y\right\vert +1$, there is an $x\in X-Y$ such that $Y\cup\{x\}\in\mathcal{F}$.
\end{definition}

Let us observe that $\{d,g\}\in\Psi(W)$, while $\{d\},\{g\}\notin\Psi(W)$,
where $W$ is the graph depicted in Figure \ref{fig101}. However, it is worth
mentioning that if $\Psi(G)$ is a greedoid and $S\in\Psi(G)$, $\left\vert
S\right\vert =k\geq2$, then according to the accessibility property, one can
build a chain
\[
\{x_{1}\}\subset\{x_{1},x_{2}\}\subset...\subset\{x_{1},...,x_{k-1}%
\}\subset\{x_{1},...,x_{k-1},x_{k}\}=S
\]
such that
\[
\{x_{1},x_{2},...,x_{j}\}\in\Psi(G)\text{, for all }j\in\{1,...,k-1\}.
\]
For example, $\{a\}\subset\{a,b\}\subset S$ is an accessibility chain of the
set $S=\{a,b,c\}\in\Psi(G_{2})$, where $G_{2}$ is presented in Figure
\ref{fig3}.

In \cite{LevMan2} it is proved the following result.

\begin{theorem}
\label{th2} For every tree $T,$ $\Psi(T)$ is a greedoid on its vertex set.
\end{theorem}

The case of bipartite graphs owning a unique cycle, whose family of local
maximum stable sets forms a greedoid is analyzed in \cite{LevMan5} (for an
example, see the graph $G_{1}$ from Figure \ref{fig3}). In general, local
maximum stable sets of bipartite graphs were treated in \cite{LevMan45}, while
for triangle-free graphs we refer the reader to \cite{LevMan07} for details.
Nevertheless, there exist non-bipartite and also non-triangle-free graphs
whose families of local maximum stable sets form greedoids. For instance, the
families $\Psi(G_{2}),\Psi(G_{3}),\Psi(G_{4})$ of the graphs in Figure
\ref{fig3} are greedoids.\begin{figure}[h]
\setlength{\unitlength}{1.0cm} \begin{picture}(5,1.8)\thicklines
\multiput(1,0.5)(1,0){4}{\circle*{0.29}}
\multiput(2,1.5)(1,0){3}{\circle*{0.29}}
\put(1,0.5){\line(1,0){3}}
\put(2,0.5){\line(0,1){1}}
\put(2,1.5){\line(1,0){1}}
\put(3,0.5){\line(1,1){1}}
\put(3,0.5){\line(0,1){1}}
\put(2.5,0){\makebox(0,0){$G_{1}$}}
\multiput(5,0.5)(1,0){3}{\circle*{0.29}}
\multiput(6,1.5)(1,0){2}{\circle*{0.29}}
\put(5,0.5){\line(1,0){2}}
\put(6,0.5){\line(1,1){1}}
\put(6,0.5){\line(0,1){1}}
\put(7,0.5){\line(0,1){1}}
\put(5,0.8){\makebox(0,0){$a$}}
\put(6.3,1.5){\makebox(0,0){$b$}}
\put(7.3,1.5){\makebox(0,0){$c$}}
\put(6,0){\makebox(0,0){$G_{2}$}}
\multiput(8,0.5)(1,0){3}{\circle*{0.29}}
\multiput(9,1.5)(1,0){2}{\circle*{0.29}}
\put(8,0.5){\line(1,0){2}}
\put(9,0.5){\line(0,1){1}}
\put(9,0.5){\line(1,1){1}}
\put(9,1.5){\line(1,0){1}}
\put(10,0.5){\line(0,1){1}}
\put(9,1.5){\line(1,-1){1}}
\put(9,0){\makebox(0,0){$G_{3}$}}
\multiput(11,0.5)(1,0){3}{\circle*{0.29}}
\multiput(11,1.5)(1,0){3}{\circle*{0.29}}
\put(11,0.5){\line(1,0){2}}
\put(11,1.5){\line(1,0){2}}
\put(11,0.5){\line(0,1){1}}
\put(11,0.5){\line(1,1){1}}
\put(11,1.5){\line(1,-1){1}}
\put(12,0.5){\line(0,1){1}}
\put(12,0.5){\line(1,1){1}}
\put(12,1.5){\line(1,-1){1}}
\put(13,0.5){\line(0,1){1}}
\put(12,0){\makebox(0,0){$G_{4}$}}
\end{picture}\caption{Graphs whose family of local maximum stable sets {form
greedoids}.}%
\label{fig3}%
\end{figure}
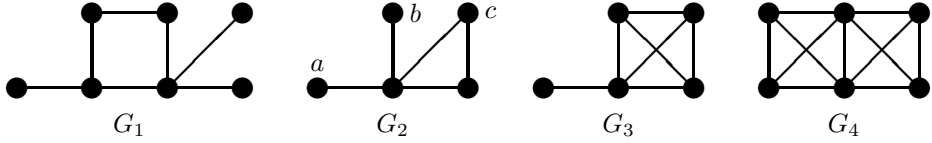

In this note we present "\textit{if and only if}" conditions for $\Psi(G)$ to
be a greedoid, where $G$ is the disjoint union, or the Zykov sum, or the
corona of a family of graphs.

\section{Disjoint union and Zykov sum of graphs}

Let $G$ be the disjoint union of the family of graphs $\{G_{i}:1\leq i\leq
p\},p\geq2$, i.e.,
\begin{align*}
V(G)  &  =V(G_{1})\cup V(G_{2})\cup...\cup V(G_{p})\text{ and}\\
E(G)  &  =E(G_{1})\cup E(G_{2})\cup...\cup E(G_{p}),
\end{align*}
under the assumption that $V(G_{i})\cap V(G_{j})=\varnothing,1\leq i<j\leq p$.
Clearly, $\alpha(G)=\alpha(G_{1})+\alpha(G_{2})+...+\alpha(G_{p})$ and
$S\subseteq V(G)$ is stable if and only if every $S\cap V(G_{i}),1\leq i\leq
p$, is stable. Moreover, one can easily prove the following result.

\begin{proposition}
\label{prop1}If $G$ is the disjoint union of the family of graphs
$\{G_{i}:1\leq i\leq p\},p\geq2$, then:

\emph{(i)} $S\in\Psi(G)$ if and only if $S\cap V(G_{i})\in\Psi(G),1\leq i\leq
p$;

\emph{(ii)} $\Psi(G)$ is a greedoid if and only if every $\Psi(G_{i}),1\leq
i\leq p$, is a greedoid.
\end{proposition}

Recall that the \textit{Zykov sum} of the graphs $G_{i},1\leq i\leq p,p\geq2$,
is the graph $Z=Z[G_{1},...,G_{p}]=G_{1}+G_{2}+...+G_{p}$ having
\begin{align*}
V(Z)  &  =V(G_{1})\cup...\cup V(G_{p}),\\
E(Z)  &  =E(G_{1})\cup...\cup E(G_{p})\cup\{v_{i}v_{j}:v_{i}\in V_{i},v_{j}\in
V_{j},1\leq i<j\leq p\}.
\end{align*}
\begin{figure}[h]
\setlength{\unitlength}{1cm}\begin{picture}(5,2.5)\thicklines
\multiput(3,0)(0,1){2}{\circle*{0.29}}
\multiput(5,0)(0,1){3}{\circle*{0.29}}
\put(3,0){\line(0,1){1}}
\put(3,0){\line(1,0){2}}
\put(3,0){\line(2,1){2}}
\put(3,0){\line(1,1){2}}
\put(3,1){\line(2,-1){2}}
\put(3,1){\line(1,0){2}}
\put(3,1){\line(2,1){2}}
\put(5,0){\line(0,1){2}}
\put(2,1){\makebox(0,0){$G_1$}}
\multiput(8,0)(0,1){3}{\circle*{0.29}}
\multiput(10,0)(0,1){3}{\circle*{0.29}}
\put(8,0){\line(0,1){2}}
\put(8,0){\line(1,0){2}}
\put(8,0){\line(2,1){2}}
\put(8,0){\line(1,1){2}}
\put(8,1){\line(2,-1){2}}
\put(8,1){\line(1,0){2}}
\put(8,1){\line(2,1){2}}
\put(8,2){\line(2,-1){2}}
\put(8,2){\line(1,-1){2}}
\put(8,2){\line(1,0){2}}
\put(10,0){\line(0,1){2}}
\put(7,1){\makebox(0,0){$G_2$}}
\end{picture}\caption{$G_{1}=Z[K_{2},P_{3}]$ and $G_{2}=Z[P_{3},P_{3}]$.}%
\label{fig11}%
\end{figure}
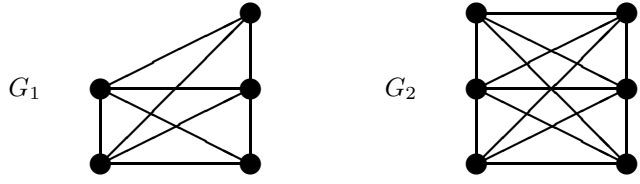Clearly, $\alpha(Z)=\max\{\alpha(G_{i}):1\leq i\leq p\}$. If all
$G_{1},G_{2},...,G_{p},p\geq2$, are complete graphs, then $Z$ is complete, as
well. In this case, we have
\[
\Psi(Z)=\Omega(Z)=\{\{v\}:v\in V(Z)\}
\]
and $\Psi(Z)$ is, evidently, a greedoid.

\begin{lemma}
\label{lem4}If $Z=Z[G_{1},...,G_{p}]$, then
\[
\min\{\left\vert S\right\vert :S\in\Psi(Z)\}\geq\max\nolimits_{2}%
\{\alpha(G_{i}):1\leq i\leq p\},
\]
where $\max_{2}\{\alpha_{i}:1\leq i\leq p\}$ is a second largest number of the sequence.
\end{lemma}

\begin{proof}
Notice that if $S\subseteq V(Z)$ is stable, then there is some $i\in
\{1,2,...,p\}$ such that $S\subseteq V(G_{i})$. Hence, if $S\in\Psi(Z)$, then
$S\in\Psi(G_{k})$ for some $k\in\{1,2,...,p\}$, and, in addition,
\[
\left\vert S\right\vert \geq\max\{\alpha(G_{i}):1\leq i\leq p,i\neq k\}.
\]

Since
\[
\max\nolimits_{2}\{\alpha(G_{i}):1\leq i\leq p\}=\min_{1\leq k\leq p}\left(
\max\{\alpha(G_{i}):1\leq i\leq p,i\neq k\}\right)  ,
\]

we get that%
\[
\min\{\left\vert S\right\vert :S\in\Psi(Z)\}\geq\max\nolimits_{2}%
\{\alpha(G_{i}):1\leq i\leq p\},
\]
which completes the proof.
\end{proof}

Let us observe that for the graphs $G_{1}=Z[K_{2},P_{3}]$ and $G_{2}%
=Z[P_{3},P_{3}]$ (depicted in Figure \ref{fig11}), $\Psi(G_{1})$ is a
greedoid, while $\Psi(G_{2})$ is not a greedoid, because $\{v\}\notin
\Psi(G_{2})$, for every $v\in V(G_{2})$.

\begin{proposition}
\label{prop2}Let $Z=Z[G_{1},...,G_{p}]$ be such that $\alpha(Z)>1$. Then
$\Psi(Z)$ is a greedoid if and only if the following assertions are true:

\emph{(i)} all $\Psi(G_{i}),1\leq i\leq p$, are greedoids;

\emph{(ii)} there is a unique $k\in\{1,2,..,p\}$ such that $G_{k}$ is not complete;

\emph{(iii)} $\Psi(Z)=\Psi(G_{k})$.
\end{proposition}

\begin{proof}
Taking into account the definition of $Z$, it follows that at least one of the
graphs $G_{i}$ is not complete, because and $\alpha(Z)>1$.

Assume that $\Psi(Z)$ is a greedoid and let $\{a\}\in\Psi(Z)$. Hence we infer
that
\[
\min\{\left\vert S\right\vert :S\in\Psi(Z)\}=1.
\]
Consequently, by Lemma \ref{lem4}, we get $1\geq\max\nolimits_{2}%
\{\alpha(G_{i}):1\leq i\leq p\}$.

Thus all $G_{i},1\leq i\leq p$ but one must be complete graphs. Suppose
$G_{k}$ is the unique non-complete graph. Then $a\in V(G_{k})$ and
$\alpha(Z)=\alpha(G_{k})$.

Clearly, all $\Psi(G_{i}),1\leq i\leq p,i\neq k$, are greedoids. In addition,
$\{v\}\notin\Psi(Z)$, for every $v\in V(Z)-V(G_{k})$, because $V(G_{k}%
)\subseteq N_{Z}(v)$ and $\alpha(G_{k})>1$. It follows that $S\subseteq
V(G_{k})$, for every $S\in\Psi(Z)$. Moreover, one can say that $S\in\Psi
(G_{k})$, i.e., $\Psi(Z)\subseteq\Psi(G_{k})$.

Otherwise, if some $A\in\Psi(Z)$ does not belong to $\Psi(G_{k})$, it follows
that there is a stable set $B$ in $N_{G_{k}}[A]$, larger than $A$. Since $B$
is stable in $Z$, as well, and $B\subseteq N_{G_{k}}[A]\subseteq N_{Z}[A]$, it
implies $A\notin\Psi(Z)$, in contradiction with the choice of $A$. On the
other hand, taking into account that no stable set in $Z$ can meet both
$V(G_{k})$ and $V(Z)-V(G_{k})$, it follows that $\Psi(G_{k})\subseteq\Psi(Z)$
is true, as well. In other words, we infer that $\Psi(Z)=\Psi(G_{k})$, which
ensures that $\Psi(G_{k})$ is a greedoid.

The converse is clear.
\end{proof}

\section{Corona of graphs}

Let $X$ be a graph with $V(X)=\{v_{i}:1\leq i\leq n\}$, and $\{H_{i}:1\leq
i\leq n\}$ be a family of graphs. Joining each $v_{i}\in V(X)$ to all the
vertices of $H_{i}$, we obtain a new graph, which we denote by $G=X\circ
\{H_{1},H_{2},...,H_{n}\}$ (see Figure \ref{fig12} for an example, where
$X=K_{3}+v_{3}v_{4}$). If $H_{1}=H_{2}=...=H_{n}=H$, we write $G=X\circ H$,
and in this case, $G$ is called the \textit{corona} of $X$ and $H$%
.\begin{figure}[h]
\setlength{\unitlength}{1cm}\begin{picture}(5,2.5)\thicklines
\multiput(6,0.5)(1,0){4}{\circle*{0.29}}
\put(6,2.5){\circle*{0.29}}
\multiput(5,1.5)(0,1){2}{\circle*{0.29}}
\multiput(7,1.5)(1,0){4}{\circle*{0.29}}
\multiput(8,2.5)(1,0){2}{\circle*{0.29}}
\multiput(6,0.5)(1,0){3}{\line(1,0){1}}
\multiput(7,0.5)(1,0){2}{\line(0,1){1}}
\multiput(7,0.5)(1,0){2}{\line(1,2){1}}
\multiput(7,1.5)(1,0){2}{\line(1,1){1}}
\multiput(8,0.5)(1,0){2}{\line(1,1){1}}
\put(5,1.5){\line(1,1){1}}
\put(5,1.5){\line(0,1){1}}
\put(5,2.5){\line(1,0){1}}
\put(5,1.5){\line(1,-1){1}}
\put(5,2.5){\line(1,-2){1}}
\put(6,0.5){\line(0,1){2}}
\put(9,1.5){\line(0,1){1}}
\qbezier(6,0.5)(7,-0.7)(8,0.5)
\put(7.8,1.3){\makebox(0,0){$x$}}
\put(6.35,2.5){\makebox(0,0){$y$}}
\put(9,1.15){\makebox(0,0){$z$}}
\put(6.75,1.3){\makebox(0,0){$u$}}
\put(10,1.85){\makebox(0,0){$t$}}
\put(6,0.1){\makebox(0,0){$v_1$}}
\put(7,0.17){\makebox(0,0){$v_2$}}
\put(8.2,0.1){\makebox(0,0){$v_3$}}
\put(9.2,0.1){\makebox(0,0){$v_4$}}
\put(4.5,2){\makebox(0,0){$K_3$}}
\put(7.1,2.1){\makebox(0,0){$K_2$}}
\put(9.3,2){\makebox(0,0){$P_3$}}
\put(10.45,1.5){\makebox(0,0){$K_1$}}
\put(3.3,1.5){\makebox(0,0){$G$}}
\end{picture}\caption{$G=(K_{3}+v_{3}v_{4})\circ\{K_{3},K_{2},P_{3},K_{1}\}$
is a well-covered graph.}%
\label{fig12}%
\end{figure}
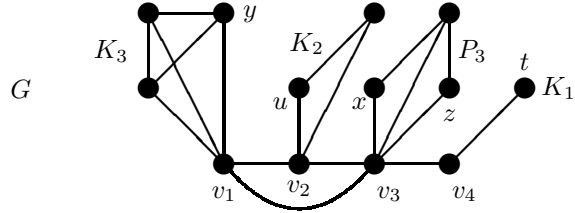

Let us notice that $G=X\circ\{H_{1},H_{2},...,H_{n}\}$ has $\alpha
(G)=\alpha(H_{1})+\alpha(H_{2})+...+\alpha(H_{n})$.

Let us consider the graph $G$ depicted in Figure \ref{fig122}. Notice that:

\begin{itemize}
\item $S_{1}=\{x,z,v_{4}\}\in\Psi(G)$ and also $S_{1}\cap V(H)\in\Psi(H)$, for
every $H\in\{K_{3},K_{2},P_{3},K_{1}\}$;

\item the set $S_{2}=\{y,v_{2}\}$ is stable, but $S_{2}\notin\Psi(G)$, because
$\{y,u,v_{3}\}\subseteq N_{G}[S_{2}]$ and it is stable and larger than $S_{2}$;

\item $\{v_{4}\},\{v_{2},v_{4}\}\in\Psi(K_{3}+v_{3}v_{4})$, but $\{v_{4}%
\},\{v_{2},v_{4}\}\notin\Psi(G)$;

\item $\{y,v_{4}\}\notin\Psi(G)$, since $\{y,t,v_{3}\}\subseteq N_{G}%
[\{y,v_{4}\}]$ and it is stable and larger than $\{y,v_{4}\}$;

\item $\{y\}\in\Psi(K_{3}),\{x,z\}\in\Psi(P_{3})$ and also $\{x,y,z\}\in
\Psi(G)$;

\item the set $S_{3}=\{y,v_{3}\}$ is stable and $S_{2}\cap V(H)\in\Psi(H)$,
for each $H\in\{K_{3},K_{2},P_{3},K_{1}\}$, but $S_{3}\notin\Psi(G)$.
\end{itemize}

\begin{lemma}
\label{lem3}Let $G=X\circ\{H_{1},H_{2},...,H_{n}\}$, where $V(X)=\{v_{i}:1\leq
i\leq n\},n\geq2$. Then the following assertions are true :

\emph{(i)} $\Psi(H_{i})\subseteq\Psi(G),1\leq i\leq n$;

\emph{(ii)} if $v_{i}\in S\in\Psi(G)$, then $H_{i}$ is complete, and $S\cap
V(H_{k})\neq\varnothing$, for each $v_{k}\in N_{X}(v_{i})$;

\emph{(iii)} if $S\in\Psi(G)$, then $S\cap V(H_{i})\in\Psi(H_{i}),1\leq i\leq
n$;

\emph{(iv)} if $S$ is a stable set in $G$ such that: $S\cap V(H_{i})\in
\Psi(H_{i}),1\leq i\leq n$, and for every $v_{i}\in S\cap V(X),H_{i}$ is a
complete graph, while $S\cap V(H_{k})\neq\varnothing$, for all $v_{k}\in
N_{X}(v_{i})$, then $S\in\Psi(G)$.
\end{lemma}

\begin{proof}
\emph{(i)} Let $A\in\Psi(H_{i})$. Then, $N_{G}[A]=N_{H_{i}}[A]\cup\{v_{i}\}$
and, thus, $A$ is a maximum stable set in $N_{G}[A]$, as well, i.e., $A\in
\Psi(G)$. Consequently, $\Psi(H_{i})\subseteq\Psi(G)$ for each $i\in
\{1,2,...,n\}$.

\emph{(ii)} If $v_{i}\in S$ and there are non-adjacent vertices $x,y\in
V(H_{i})$, then the set $S\cup\{x,y\}-\{v_{i}\}$ is stable in $N_{G}[S]$,
larger than $S$, in contradiction with $S\in\Psi(G)$. Therefore, $H_{i}$ must
be a complete graph.

Assume that $S\cap V(H_{k})=\varnothing$ for some $v_{k}\in N_{X}(v_{i})$.

If $N(v_{k})\cap S=\{v_{i}\}$, then for every $x\in V(H_{i})$, the set
$S\cup\{v_{k},x\}-\{v_{i}\}$ is stable in $N_{G}[S]$ and larger than $S$, in
contradiction with $S\in\Psi(G)$.

If $N(v_{k})\cap S=\{v_{i},v_{j_{1}},v_{j_{2}},...,v_{j_{q}}\}$, then
\[
S\cup\{x,v_{k}\}\cup\{x_{j_{1}},x_{j_{2}},...,x_{j_{q}}\}-\{v_{i},v_{j_{1}%
},v_{j_{2}},...,v_{j_{q}}\}
\]
is a stable set in $N_{G}[S]$ for every $x\in V(H_{i})$ and each $x_{j_{t}}\in
V(H_{j_{t}}),1\leq t\leq q$, larger than $S$, in contradiction with $S\in
\Psi(G)$.

Consequently, $S\cap V(H_{k})\neq\varnothing$, for every $v_{k}\in N_{X}%
(v_{i})$.

\emph{(iii)} Assume that $S\in\Psi(G)$.

If $S_{j}=S\cap V(H_{j})\notin\Psi(H_{j})$, then $v_{j}\notin S$ (because
$S_{j}\neq\varnothing$) and there is some stable set $A_{j}\subseteq N_{H_{j}%
}[S_{j}]$ larger than $S_{j}$. Since
\[
N_{H_{j}}[S_{j}]\cup\{v_{j}\}=N_{G}[S_{j}]
\]
and $v_{j}\notin A$, we get that $(S-S_{j})\cup A$ is a stable set included in
$N_{G}[S]$ and $\left\vert S\right\vert <\left\vert (S-S_{j})\cup A\right\vert
$, in contradiction with $S\in\Psi(G)$. Therefore, $S\cap V(H_{i})\in
\Psi(H_{i})$ for every $i\in\{1,2,...,n\}$.

\emph{(iv)} We have to prove that $\left\vert A\right\vert \leq\left\vert
S\right\vert $ for every stable set $A\subseteq N_{G}[S]$.

Let us define the following partitions of the sets $A$ and $S$:%
\[
A=A_{1}\cup A_{2}\cup A_{3},\quad S=S_{1}\cup S_{2}\cup S_{3}%
\]
where%
\[
A_{1}=%
{\textstyle\bigcup\limits_{v_{j}\in\left(  A-S\right)  }}
\left[  A\cap\left(  V(H_{j})\cup\left\{  v_{j}\right\}  \right)  \right]
,\quad S_{1}=%
{\textstyle\bigcup\limits_{v_{j}\in\left(  A-S\right)  }}
\left[  S\cap\left(  V(H_{j})\cup\left\{  v_{j}\right\}  \right)  \right]
\]%
\[
A_{2}=%
{\textstyle\bigcup\limits_{v_{j}\in V(X)-\left(  A\cup S\right)  }}
\left[  A\cap\left(  V(H_{j})\cup\left\{  v_{j}\right\}  \right)  \right]
,\quad S_{2}=%
{\textstyle\bigcup\limits_{v_{j}\in V(X)-\left(  A\cup S\right)  }}
\left[  S\cap\left(  V(H_{j})\cup\left\{  v_{j}\right\}  \right)  \right]
\]%
\[
A_{3}=%
{\textstyle\bigcup\limits_{v_{j}\in S}}
\left[  A\cap\left(  V(H_{j})\cup\left\{  v_{j}\right\}  \right)  \right]
,\quad S_{3}=%
{\textstyle\bigcup\limits_{v_{j}\in S}}
\left[  S\cap\left(  V(H_{j})\cup\left\{  v_{j}\right\}  \right)  \right]  .
\]

Our intent is to show that
\[
\left\vert A_{k}\right\vert \leq\left\vert S_{k}\right\vert ,1\leq k\leq3,
\]
which will lead us to the conclusion that $\left\vert A\right\vert
\leq\left\vert S\right\vert $.

\textit{Case 1.} $v_{j}\in A-S$.

Since $S$ is maximal in $N_{G}[S]$, we infer that $v_{j}y\in E(G)$, for some
$y\in S$. If $y\in V(H_{j})$, then $S\cap V(H_{j})\neq\varnothing$. Otherwise,
$y\in V(X)$ and according with the hypothesis on $S$, again $S\cap
V(H_{j})\neq\varnothing$. Therefore, we get that
\[
\left\vert A\cap\left(  V(H_{j})\cup\left\{  v_{j}\right\}  \right)
\right\vert =1\leq\left\vert S\cap\left(  V(H_{j})\cup\left\{  v_{j}\right\}
\right)  )\right\vert ,
\]
which implies $\left\vert A_{1}\right\vert \leq\left\vert S_{1}\right\vert $.

\textit{Case 2.} $v_{j}\in V(X)-\left(  A\cup S\right)  $.

Since $S\cap V(H_{j})\in\Psi(H_{j})$, we have that $\left\vert A\cap
V(H_{j})\right\vert \leq\left\vert S\cap V(H_{j})\right\vert $. Together with
the condition $v_{j}\in V(X)-\left(  A\cup S\right)  $ it gives
\[
\left\vert A\cap\left(  V(H_{j})\cup\left\{  v_{j}\right\}  \right)
\right\vert \leq\left\vert S\cap\left(  V(H_{j})\cup\left\{  v_{j}\right\}
\right)  )\right\vert .
\]

Therefore, it follows that $\left\vert A_{2}\right\vert \leq\left\vert
S_{2}\right\vert $.

\textit{Case 3.} $v_{j}\in S$.

According with the hypothesis on $S$, $H_{j}$ is a clique. Consequently, we
obtain
\[
\left\vert A\cap\left(  V(H_{j})\cup\left\{  v_{j}\right\}  \right)
\right\vert \leq1\leq\left\vert S\cap\left(  V(H_{j})\cup\left\{
v_{j}\right\}  \right)  )\right\vert ,
\]
which ensures that $\left\vert A_{3}\right\vert \leq\left\vert S_{3}%
\right\vert $.
\end{proof}

The following theorem generalizes some partial findings from
\cite{LevMann07Buch}, \cite{LevMan08a}, \cite{LevMan08b}.

\begin{theorem}
If $G=X\circ\{H_{1},H_{2},...,H_{n}\}$ and $H_{1},H_{2},...,H_{n}$ are
non-empty graphs, then $\Psi(G)$ is a greedoid if and only if every
$\Psi(H_{i}),i=1,2,...,n$, is a greedoid.
\end{theorem}

\begin{proof}
Assume that $\Psi(G)$ is a greedoid.

According to Lemma \ref{lem3}\emph{(i)},\emph{(iii)}, we get that
\[
\Psi(H_{i})=\{S\cap V(H_{i}):S\in\Psi(G)\},1\leq i\leq n.
\]
Hence, every $\Psi(H_{i})$ satisfies both accessibility property and exchange
property, i.e., $\Psi(H_{i})$ is a greedoid.

Conversely, suppose that every $\Psi(H_{i}),1\leq i\leq n$, is a greedoid.

Firstly, we show that $\Psi(G)$ satisfies the accessibility property.

Let $S\in\Psi(G)$ and $S\neq\varnothing$.

If $v_{i}\in S\cap V(X)$, then $N_{X}(v_{i})\cap S=\varnothing,V(H_{i})\cap
S=\varnothing$, while, by Lemma \ref{lem3}\emph{(ii)}, $S\cap V(H_{k}%
)\neq\varnothing$ holds for every $v_{k}\in N(v_{i})$. Hence, we may infer
that $S-\{v_{i}\}\in\Psi(G)$.

If $S\cap V(X)=\varnothing$, then there is some $i\in\{1,2,...,n\}$, such that
$S_{i}=S\cap V(H_{i})\neq\varnothing$ and $S_{i}\in\Psi(H_{i})$, according to
Lemma \ref{lem3}\emph{(iii)}. Since $\Psi(H_{i})$ is a greedoid, there is some
$x\in S_{i}$ such that $S_{i}-\{x\}\in\Psi(H_{i})$. Since
\[
N_{G}[S-\{x\}]\cap V(H_{i})=N_{H_{i}}[S_{i}-\{x\}],
\]
while
\[
N_{G}[S-\{x\}]\cap V(H_{j})=N_{G}[S]\cap V(H_{j})
\]
\ for every $j\neq i$,\ we may conclude that $S-\{x\}\in\Psi(G)$.

We check now the exchange property. Let $S_{1},S_{2}\in\Psi(G)$ be with
$\left\vert S_{1}\right\vert =\left\vert S_{2}\right\vert +1$.

\textit{Case 1}. $S_{1}\cap V(H_{j})=S_{2}\cap V(H_{j})$ for all
$j\in\{1,2,...,n\}$. Then there is some $v_{i}\in S_{1}-S_{2}$, because
$\left\vert S_{1}\right\vert >\left\vert S_{2}\right\vert $. Hence, it follows
$S_{1}\cap V(H_{i})=\varnothing$ , which ensures that also $S_{2}\cap
V(H_{i})=\varnothing$. By Lemma \ref{lem3}\emph{(ii)}, we have that, for every
$v_{k}\in N_{G}(v_{i})$, $S_{1}\cap V(H_{k})\neq\varnothing$ which implies
that also $S_{2}\cap V(H_{k})\neq\varnothing$. Consequently, using Lemma
\ref{lem3}\emph{(iv)}, we may infer that $S_{2}\cup\{v_{k}\}\in\Psi(G)$.

\textit{Case 2.} There is some $i\in\{1,2,...,n\}$, such that $A_{1}=S_{1}\cap
V(H_{i})$ is larger than $A_{2}=S_{2}\cap V(H_{i})$. Since $A_{1},A_{2}\in
\Psi(H_{i})$ and $\Psi(H_{i})$ is a greedoid, there must exist some $x\in
A_{1}-A_{2}$, such that $A_{2}\cup\{x\}\in\Psi(H_{i})$. Hence, we get also
that $S_{2}\cup\{x\}\in\Psi(G)$.

Consequently, $\Psi(G)$\ satisfies the exchange property.

In conclusion, $\Psi(G)$\ forms a greedoid on the vertex set of $G$.
\end{proof}

\begin{corollary}
$\Psi(X\circ H)$ is a greedoid if and only if $\Psi(H)$ is a greedoid.
\end{corollary}

\section{Conclusions and future work}

Let $\{H_{1},...,H_{n}\}$ be a family of graphs indexed by the vertex set
$\{1,2,..,n\}$ of a graph $H_{0}$. The graph denoted by $H_{0}[H_{1}%
,H_{2},...,H_{n}]$ is defined as:%

\[
V(H_{0}[H_{1},H_{2},...,H_{n}])=\{1\}\times V(H_{1})\cup...\cup\{n\}\times
V(H_{n}),
\]
and $(i,x),(j,y)\in V(H_{0}[H_{1},H_{2},...,H_{n}])$ are adjacent if and only
if either \emph{(i)} $ij\in E(H_{0})$ or \emph{(ii) }$i=j$ and $xy\in
E(H_{i})$. For instance, $\overline{K_{n}}[H_{1},H_{2},...,H_{n}]$ is the
disjoint union of the graphs $H_{1},...,H_{n}$; $K_{n}[H_{1},H_{2},...,H_{n}]$
is the Zykov sum of $H_{1},...,H_{n}$; while if $H_{1}=H_{2}=...=H_{n}$, then
$H_{0}[H_{1},H_{2},...,H_{n}]$ is known as lexicographic product $H_{0}\bullet
H_{1}$. It seems to be interesting to establish necessary and sufficient
conditions ensuring that $\Psi(H_{0}[H_{1},H_{2},...,H_{n}])$\ forms a
greedoid. When $H_{0}\in\{\overline{K_{n}},K_{n}\}$, Propositions \ref{prop1},
\ref{prop2} give the conditions needed.


\begin{thebibliography}{99}                                                                                               %


\bibitem {BjZiegler}A. Bj\"{o}rner, G. M. Ziegler, \emph{Introduction to
greedoids}, in N. White (ed.), \emph{Matroid Applications}, 284-357, Cambridge
University Press, 1992.

\bibitem {KorLovSch}B. Korte, L. Lov\'{a}sz, R. Schrader, \emph{Greedoids},
Springer-Verlag, Berlin, 1991.

\bibitem {LevMan5}V. E. Levit, E. Mandrescu, \emph{Unicycle bipartite graphs
with only uniquely restricted maximum matchings}, in Proceedings of the Third
International Conference on Combinatorics, Computability and Logic, (DMTCS'1),
Springer, (C.S. Calude, M. J. Dinneen and S. Sburlan eds.) (2001) 151-158.

\bibitem {LevMan2}V. E. Levit, E. Mandrescu, \emph{A new greedoid: the family
of local maximum stable sets of a forest}, Discrete Applied Mathematics
\textbf{124} (2002) 91-101.

\bibitem {LevMan45}V. E. Levit, E. Mandrescu, \emph{Local maximum stable sets
in bipartite graphs with uniquely restricted maximum matchings}, Discrete
Applied Mathematics \textbf{132} (2003) 163-174.

\bibitem {LevMan07}V. E. Levit, E. Mandrescu, \emph{Triangle-free graphs with
uniquely restricted maximum matchings and their corresponding greedoids}%
,\emph{\ }\newline Discrete Applied Mathematics \textbf{155} (2007) 2414-2425.

\bibitem {LevMann07Buch}V. E. Levit, E. Mandrescu, \emph{On local maximum
stable sets of the corona of a path with complete graphs}, The $6^{th}$
Congress of Romanian Mathematicians, June 28 - July 4, 2007, University of
Bucharest, Bucharest, Romania.

\bibitem {LevMan08a}V. E. Levit, E. Mandrescu, \emph{Well-covered graphs and
greedoids}, Theory of Computing 2008. Proceedings of the Fourteenth Computing:
The Australasian Theory Symposium (CATS08), Wollongong, NSW. Conferences in
Research and Practice in Information Technology, J. Harland and P. Manyem,
eds., Volume \textbf{77} (2008) 89-94.

\bibitem {LevMan08b}V. E. Levit, E. Mandrescu, \emph{The clique corona
operation and greedoids}, \newline Combinatorial Optimization and
Applications, Second International Conference, COCOA 2008, Lecture Notes in
Computer Science \textbf{5165} (2008) 384-392.

\bibitem {NemhTro}G. L. Nemhauser, L. E. Trotter, Jr., \emph{Vertex packings:
structural properties and algorithms}, Mathematical Programming \textbf{8}
(1975) 232-248.
\end{thebibliography}
\end{document}